\documentclass[11pt]{amsart}
\usepackage{amssymb}

\newtheorem*{notn}{Notation}

\newtheorem{thm}{Theorem}[section]
\newtheorem{lem}[thm]{Lemma}

\newtheorem{cor}[thm]{Corollary}
\newtheorem{problem}[thm]{Problem}
\newtheorem{Def}[thm]{Definition}
\newtheorem{prop}[thm]{Proposition}
\newtheorem{rem}[thm]{Remark}
\newtheorem{ex}[thm]{Example}

\newcommand{\bdfn}{\begin{Def} \rm}
\newcommand{\edfn}{\end{Def}}
\newcommand{\tFAE}{the following are equivalent}

\newcommand{\ra}{\rightarrow}

\newcommand{\es}{\emptyset}
\newcommand{\ci}{\subseteq}

\newcommand{\al}{\alpha}

\newcommand{\de}{\delta}
\newcommand{\e}{\varepsilon}

\newcommand{\la}{\lambda}
\newcommand{\si}{\sigma}

\newcommand{\Ga}{\Gamma}

\newcommand{\mb}{\mathbb}
\newcommand{\mc}{\mathcal}
\newcommand{\sm}{\setminus}

\newcommand{\iy}{\infty}

\newcommand{\beqa}{\begin{eqnarray*}}
\newcommand{\eeqa}{\end{eqnarray*}}

\newcounter{cnt1}
\newcounter{cnt2}
\newcounter{cnt3}
\newcounter{cnt4}
\newcommand{\blr}{\begin{list}{$($\roman{cnt1}$)$} {\usecounter{cnt1}
 \setlength{\topsep}{0pt} \setlength{\itemsep}{0pt}}}
\newcommand{\blR}{\begin{list}{\Roman{cnt4}.\ } {\usecounter{cnt4}
 \setlength{\topsep}{0pt} \setlength{\itemsep}{0pt}}}
\newcommand{\bla}{\begin{list}{$(\alph{cnt2})$} {\usecounter{cnt2}
 \setlength{\topsep}{0pt} \setlength{\itemsep}{0pt}}}
\newcommand{\bln}{\begin{list}{$($\arabic{cnt3}$)$} {\usecounter{cnt3}
 \setlength{\topsep}{0pt} \setlength{\itemsep}{0pt}}}
\newcommand{\el}{\end{list}}

\begin{document}

\title[Various notions of best approximation property]{Various notions of best approximation property in spaces of Bochner integrable functions}

\author[Paul]{Tanmoy Paul}
\address[Tanmoy Paul]{Department of Mathematics,
 Indian Institute of Technology Hyderabad,
 Kandi Campus, Sangareddy, Telangana 502285, India\\
\emph{E-mail~:} {\tt tanmoy@iith.ac.in}}



\thanks{The research was supported by DST-SERB, India. Award No. MA/2013-14/003/DST/TPaul/0104.}

\begin{abstract}
We derive that for a separable proximinal subspace $Y$ of $X$, $Y$ is 
strongly proximinal (strongly ball proximinal) if and only if for $1\leq p< 
\iy$, $L_p(I,Y)$ is 
strongly proximinal (strongly ball proximinal) in $L_p(I,X)$. Case for $p=\iy$ follows from stronger assumption on $Y$ in $X$ (uniform proximinality). It is observed that for a separable proximinal subspace $Y$ in $X$, $Y$ is ball proximinal in $X$ if and only if $L_p(I,Y)$ is ball proximinal in $L_p(I,X)$ for $1\leq p\leq\iy$. Our observations also include the fact that for any (strongly) proximinal subspace $Y$ of $X$, if every separable subspace of $Y$ is ball (strongly) proximinal in $X$ then $L_p(I,Y)$ is ball (strongly) proximinal in $L_p(I,X)$ for $1\leq p<\iy$.
We introduce the notion of uniform proximinality of a closed convex set in a Banach 
space, which is wrongly defined in \cite{LZ}. Several examples are given having this property, viz. any $U$-subspace of a Banach space, closed unit ball $B_X$ of a space with $3.2.I.P$, closed unit ball of any M-ideal of a space with $3.2.I.P.$ are uniformly proximinal. A new class of examples are given having this property.
\end{abstract}

\maketitle

\section{Preliminaries and Definitions}\label{S10}

Let $X$ be a Banach space and $C$ be a closed convex 
subset of $X$. For $x\in X$, 
let $d(x,C)=\inf_{z\in C}\|x-z\|$ and 
$P_C(x)=\{z\in
C:\|x-z\|=d(x,C)\}$. The set valued mapping $P_C:X\to 2^C$ is called the metric 
projection of $C$ and the points in $P_C(x)$ are called the best approximation from $x$ in $C$.
We call the subset $C$ {\it proximinal} (or it has best approximation property) if for every point $x\in X\sm
C$, $P_C(x)\neq\es$. 

Let $(\Omega,\mc{M},\mu)$ be a finite measure space. For a 
Banach space $X$ consider 
the Banach space of Bochner $p$-integrable (essentially bounded for $p=\iy$) 
functions on $\Omega$ with values in $X$, endowed with the usual $p$-norm viz. $L_p(\Omega,X)$. Let us recall  
any such function is essentially a strongly measurable function, separably valued and if $(s_n)$ is a sequence of simple functions such that $s_n(t)\to f(t)$ a.e. then $\lim_n\int_I \|s_n(t)\|^pdm(t)=\int_I\|f(t)\|^pdm(t)$. 
In \cite{Kh, DK, LZ, Men} the authors discussed for a finite measure space  how often the property of best approximation of 
$Y$ in $X$ is stable under the spaces of functions $L_p(\Omega,Y)$ in 
$L_p(\Omega,X)$. Let us recall the following Theorem in this context.

\begin{thm}\label{T1}
Let $Y$ be a subspace of $X$ and $f\in L_p(\Omega,X)$ then, \bla
\item \cite[Theorem~5]{WL}$d(f,L_p(\Omega,Y))=\|d(f(.),Y)\|_p$ for $1\leq p\leq\iy$.
\item \cite[Theorem~3.4]{Men}For a separable subspace $Y$ of $X$, $L_p(\Omega,Y)$ is proximinal in $L_p(\Omega,X)$ if and only if $Y$ is proximinal in $X$, for $1\leq p\leq\iy$. 
\item \cite[Corollary~2]{WL}$f\in P_{L_p(\Omega,Y)}(g)$ if and only if $f(t)\in 
P_Y(g(t))$ a.e. for $1\leq p<\iy$.
\item \cite[Proposition~2.5]{Men}$L_\iy(\Omega,Y)$ is proximinal in $L_\iy(\Omega,X)$ if 
and only if for $f\in L_\iy(\Omega,X)$ there exists $g\in L_\iy(\Omega,Y)$ such that 
$f(t)\in P_Y(g(t))$ a.e.
\el
\end{thm}

Suppose $I=[0,1]$, and $(I,\mc{B},m)$ stands for the complete 
Lebesgue measure space over the Borel $\sigma$-field $\mc{B}$. 
After Saidi's paper, \cite{Sa}, people find it is worth investigating about the proximinality of closed unit ball of a proximinal subspace. The authors in \cite{BLR} investigate the proximinality of $L_p(I,B_Y)$ in $L_p(I,X)$ if $B_Y$ is proximinal in $X$. Recall the following results from \cite[Pg~$12$]{BLR}.

\begin{thm}
Let $Y$ be a separable ball proximinal subspace of $X$. Then
\bla
\item $L_\iy(I,Y)$ is ball proximinal in $L_\iy(I,X)$.
\item $L_p(I,B_Y)$ is proximinal in $L_p(I,X)$.
 \el
\end{thm}

A latest article in this context is \cite{LZ}. It is also relevant to mention here that for a proximinal subspace 
$Y$, $L_1(I,Y)$ is not necessarily proximinal in $L_1(I,X)$ if $Y$ is not separable \cite{Men}. Light and Cheney also discussed about this best approximation property in the function spaces of type $L_p(\Omega, X)$ in \cite[Chapter~$2$]{LC}. Discussion in \cite[Chapter~$10$]{LC} is also relevant to the content of this paper. Our aim in this paper is to study various strengthenings of best approximation property, defined in Definition~\ref{D1}, of $L_p(I,Y)$ in $L_p(I,X)$. A concise presentation of this work is available in Section~\ref{S11}.

We now state few known Definitions from the literature which are relevant and also have impacts to the main theme of this paper. First recall from \cite{BLR,GI} the following stronger versions of proximinality.

\bdfn \label{D1} 
\bla
\item A closed convex subset $C$ of $X$ is {\it Strongly
proximinal} if it is proximinal and for a given $x\in X\sm C$ and $\e>0$ there
exists a $\de>0$ such that
$P_C(x,\de)\ci P_C(x)+\e B_X$,
where $P_C(x,\de)=\{z\in C:\|x-z\|\leq d(x,C)+\de\}$.

\item A subspace $Y$ is said to be {\it Ball proximinal} if $B_Y$ is proximinal in $X$.

\item A subspace $Y$ is said to be {\it Strongly ball proximinal} if $B_Y$ is strongly proximinal.
\el
\edfn

Readers can come across the articles \cite{BLR,DN,GI} for various examples of 
subspaces having these proximity properties. 

Recall the following notions for a set valued map.

\bdfn
If $T$ is
a topological space, then a set-valued map $\Ga: T\to 2^X$ is said to be
\bla
\item upper semi-continuous, abbreviated usc (resp. lower semi-continuous, abbreviated lsc) if for any neighborhood $\mc{U}$ of $\Ga(t)$ there exists a neighborhood $V$ of $t$ such that for all $s\in V, \Ga(s)\ci \mc{U}$ (if for $x\in \Ga(t)$ any sequence $t_n\to t$ there exists a sequence $x_n\in \Ga(t_n)$ converging to $x$).

\item upper Hausdorff semi-continuous, abbreviated uHsc. (resp. lower Hausdorff
semi-continuous, abbreviated lHsc) if for every $t\in T$ and every $\e> 0$,
there is a neighborhood $N$ of $t$, such that $\Ga(t)\ci\Ga(t_0)+ \e B_X$ (resp.
$\Ga(t_0)\ci\Ga(t)+ \e B_X$) for each $t\in N$. 

\item Hausdorff continuous, abbreviated H-continuous, if it is both uHsc and 
lHsc. 
\el
\edfn

From the definition of strong proximinality, it is clear that if $Y$ is a 
strongly proximinal subspace then $P_Y$ is uHsc. In general we have usc $\Rightarrow$ uHsc and lHsc $\Rightarrow$ lsc and if the above $\Ga$ is compact valued then usc $\Leftrightarrow$ uHsc and lHsc $\Leftrightarrow$ lsc.

The following notion was introduced by Yost in \cite{Y}. The author established some connections between the properties of best approximation and the following for a subspace of a Banach space.

\bdfn\label{D4}\cite{Y}
A subspace $Y$ of a Banach space $X$ is said to have the $1\frac{1}{2}$ -ball property if, whenever $\|x-y\|<r+s$ where $y\in Y$ and $x\in X$ with $B[x,r]\cap Y\neq\es$
then $B[x,r]\cap B[y,s]\cap Y\neq\es$.
\edfn

It is well known that a subspace $Y$ having $1\frac{1}{2}$ ball property is strongly proximinal. There are many function spaces and function algebras in the class of continuous functions having this property. 

Let us recall the following notion from \cite{Lim}.

\bdfn\label{D5}
A Banach space $X$ is said to have $3.2.I.P.$ if for for any three closed balls in $X$ which are pairwise intersecting actually intersect in $X$.
\edfn

Lindenstrauss monograph \cite{Lin} was the first where the above property was appeared for the first time, although the article \cite{Lim} by Lima encounters a systematic study of intersection properties of balls in Banach spaces.

\section{Main results and subsequent discussion}\label{S11}

The following problems are the origin of this investigation.

\begin{problem}\label{Q1}
 Let $Y$ be a subspace of $X$ which is strongly proximinal (ball proximinal). Is $L_p(\Omega,Y)$ strongly proximinal (ball proximinal) in $L_p(\Omega,X)$ for $1\leq p\leq\iy$ $?$
\end{problem}

The above problem on ball proximinality is asked in \cite[Pg 12]{BLR}.

\begin{problem}\label{Q2}
Let $f\in L_p(\Omega,X)$ and $Y$ be a subspace of $X$. What is the numerical value of $d(f,B_{L_p(\Omega,Y)})$ $?$
\end{problem}

\begin{problem}\label{Q3}
 Let $Y$ be a subspace of $X$ having $1\frac{1}{2}$ ball property and $(\Omega,\mc{M},\mu)$ be a finite measure space. Does $L_p(\Omega,Y)$ has $1\frac{1}{2}$ ball property in $L_p(\Omega,X)$ for $p=1,\iy$ $?$
\end{problem}

Remark~\ref{R3} states if $L_\iy(\Omega,Y)$ is strongly proximinal in $L_\iy(I,X)$ then $P_Y$ must be lHsc, on the other $Y$ would be strongly proximinal in $X$ for the same. Hence $P_Y$ is Hausdorff continuous if $L_\iy(\Omega,Y)$ is strongly proximinal in $L_\iy(I,X)$. Hence it raises the following question.

\begin{problem}\label{Q4}
Let $P_Y:X\ra 2^Y$ be Hausdorff continuous. Then what is the appropriate condition on $Y$ in $X$ which makes $L_\iy(\Omega,Y)$ strongly proximinal in  $L_\iy(\Omega,X)$ and vice versa $?$
\end{problem}

We considered these problems for the measure space $(I,\mc{B},m)$. The results in Section~\ref{S4} only require that the measure space has to be positive with total variation $1$, the other results can be derived for any finite measure space. The main results in this article are the following:
\begin{thm}[Theorem~\ref{T6},\ref{T12}]\label{M1}
For a separable proximinal subspace $Y$ of $X$, $Y$ is strongly proximinal (strongly ball proximinal)
in $X$ if and only if $L_p(I,Y)$ is strongly proximinal (strongly ball proximinal)
in $L_p(I,X)$, for $1\leq p<\iy$.
\end{thm}

\begin{thm}[Theorem~\ref{T9}]\label{M2}
For a separable proximinal subspace $Y$ of $X$, $Y$ is ball proximinal 
in $X$ if and only if $L_p(I,Y)$ is ball 
proximinal in $L_p(I,X)$, for $1\leq 
p\leq \iy$.
\end{thm}

And also,
\begin{thm}[Theorem~\ref{T14}]\label{M3}
Let $Y$ be a separable proximinal subspace of $X$, then consider the following statements. \bla
\item $Y(B_Y)$ is uniformly proximinal in $X$.
\item $L_\iy(I,Y)(B_{L_\iy(I,Y)})$ is uniformly proximinal in $L_\iy(I,X)$.
\item $L_\iy(I,Y)(B_{L_\iy(I,Y)})$ is strongly proximinal in $L_\iy(I,X)$.
\el
Then $(a)\Longleftrightarrow (b)$ and $(b)\Longrightarrow (c)$.
\end{thm}

We couldn't answer the Problem~\ref{Q4}, the above Theorem is a partial answer of Problem~\ref{Q4}.
A section-wise illustration of this work is outlined in the next few paragraphs.

In Section~\ref{S2} we discuss some distance formulas which enable us to conclude 
the 
strong proximinality of $L_p(I,Y)$ in $L_p(I,X)$. These distance formulas are 
proved with the help of pathologies of measurable set valued functions and 
their measurable selections. Problem~\ref{Q3} is answered in Theorem~\ref{T4}.

The non-availability of conclusion in Theorem~\ref{M1} for $p=\iy$ invites a uniform version of strong proximinality of $Y$ in $X$, as discussed in Section~\ref{S3}. 
To begin with, the content of Section~\ref{S3} we would like to thank the authors in 
\cite{LZ} for drawing our attention towards the notion of 'uniform 
proximinality' in Banach space. However, a similar notion dates back to the paper by Pai and 
Nowroji (\cite{PN}) in the context of Property-$(R_2)$; nevertheless, the way 
used in \cite[Pg~79]{LZ} to define 'uniform proximinality' is wrong. A simple geometry 
in the Euclidean space $\mb{R}^2$ clarifies the flaw (Example~\ref{E5}).

We adopt the idea introduced in \cite{PN} in terms of Property-($R_2$) and 
define 'uniform proximinality' of a closed convex set. Section~\ref{S3} is devoted to discussing this property. Strong proximinality can now be viewed as a local version of this 'uniform proximinality'. Several examples are given which satisfy this 
property; the list includes closed convex subsets of uniformly convex 
space, subspace with $1\frac{1}{2}$-ball property and any $U$-proximinal subspace (see \cite{Lau}). An elegant observation in this context is that closed unit ball of a Banach space is not necessarily uniformly proximinal (using Example in \cite{La}), we derive that it is true if $X$ has $3.2.I.P$ (see \cite{Lim}). Finally, we prove
the strong proximinality of $L_\iy(I,Y)$ in $L_\iy(I,X)$ as a necessary condition for uniform proximinality of $Y$ in $X$ (Theorem~\ref{M3}).
A weaker version of \cite[Theorem~15]{Rao} is also proved here.

Section~\ref{S4} is devoted to ball proximinality and strong ball proximinality of 
$L_p(I,Y)$ in $L_p(I,X)$. One can define $L_p(I,B_X)$, similar to the space $L_p(I,X)$, which represents the set of measurable functions from $I$ to $B_X$ 
which are $p$-integrable. It is proved for $f\in L_p(I,X)$, 
$d(f,L_p(I,B_Y))=d(f,B_{L_p(I,Y)})$ for $1\leq p\leq\iy$ which answers Problem~\ref{Q2}. This result together 
with Theorem~\ref{T13} leads to some interesting observations. The main results 
in this Section are stated in Theorem~\ref{M2}. Our results answer the question 
raised in \cite{BLR} after Theorem~4.10. 

Since in a Banach space $X$, $B_X$ is not necessarily 
strongly proximinal in $X$ we found it is meaningful to identify some cases when the answer is affirmative. From \cite{DS} it follows that $B_{L_p(\mu)}$ is 
strongly proximinal in $L_p(\mu)$ (spaces having reflexivity and Kadec-Klee property) for any positive measure $\mu$ when $1<p<\iy$. 
From our result it follows that the conclusion is still true for $L_p(\mu)$ where $p=1, \iy$ (for real scalar); 
in fact the result holds true for $B_{L_p(I,X)}, 1\leq p\leq\iy$ when and only 
when $X$ has the similar property.

A new class of examples is given in Section~\ref{S5} which are uniformly proximinal.

For a Banach space $X, B_X, S_X$ and $B[x,r]$ denote the closed unit ball, the
closed unit sphere and closed ball with centre at $x$ and radius $r$
respectively. All Banach spaces are assumed to be complex unless otherwise stated. Those spaces that have any intersection properties of balls like $3.2.I.P.$, $4.2.I.P.$ are assumed to be real.
$X$ will always denote a Banach space and by a subspace we always
mean a closed subspace. 

\section{Strong proximinality of $L_p(I,Y)$ in $L_p(I,X)$}\label{S2}

Similar to the Theorem~\ref{T1} we now approach towards a distance formula which 
is actually stated in Theorem~\ref{L3}. To this end we need the following 
pathologies related to the set valued functions which help us to derive Theorem~\ref{L3}. 

\begin{lem}\label{L1}
\bla
\item Let $X$ be a Banach space and $Y$ be a proximinal subspace of $X$ such that the 
metric projection $P_Y$ is uHsc. Then the mapping $G:X\times 
X\to \mb{R}$ defined by  $G((x,z))=d(x,P_Y(z))$ is upper semi-continuous in 
first variable and lower semi-continuous in second variable.
\item Let $Y$ be a subspace as defined in $(a)$ and is also separable, then for 
any two measurable
functions $f:I\to Y$ and $g:I\to X$ the mapping $\varphi:I\to \mb{R}$ defined by
$\varphi(t)=d(f(t), P_Y(g(t)))$ is measurable.
\el
\end{lem}
\begin{proof}
$(a).$ Upper semi continuity of $G$ at it's first variable follows from the fact that, 
for a closed set $A$ if $h(x)=d(x,A)$ then $h$ defines a continuous (and hence upper semi-continuous) mapping from $X$ to 
$\mb{R}$.

On the other hand let $\e >0$. Since $P_Y$ is uHsc, there exists a 
$\de>0$ such that
 $P_Y(z)\subseteq P_Y(z_0)+ \e B_Y$ whenever $\|z-z_0\|<\de$. If $(z_n)$ 
converges to $z$, there exists an $N\in \mathbb{N}$ such that $\|z_n - z\| < 
\de$ for all $n\ge N$. Hence for $n\ge N$ we get,
$  d(x, P_Y(z_n)) \ge d(x, P_Y(z)+\e B_Y) \ge d(x, P_Y(z))- \e.$

Hence we have $\liminf_n d(x,P_Y(z_n))\ge d(x, P_Y(z))$.

$(b).$ Let $D\ci Y$ be a countable dense subset of $Y$. It is clear that the 
mapping $A:I\to Y\times X$ defined by
$A(t)=(f(t),g(t))$ is measurable. We now show that $G:Y\times X\to\mb{R}$
defined by $G((y,x))=d(y,P_Y(x))$ is measurable. Hence $\varphi(t)=G(A(t))$ will
be measurable.

To this end we show that $G^{-1}([\al,\iy))$ is measurable for all real 
$\al$'s. 

Now, $G((y,x))\geq  \al\Longleftrightarrow$ 

$(\forall n\in\mb{N})(\exists z_n\in 
D)\left[\|y-z_n\|<\frac{1}{n} ~\&~ G((z_n,x))>\al- \frac{1}{n}\right]\Longleftrightarrow$ 

$(y,x)\in \bigcap_n\bigcup_{z\in D}\left[\{y\in
Y:\|y-z\|<\frac{1}{n}\}\times\{x\in X:G((z,x))>\al-\frac{1}{n}\}\right]$.

Clearly if $(y,x)\in$ RHS, then there exists a sequence $(z_n)\ci D$ such
that $G((z_n,x))>\al+\frac{1}{n}$ and $z_n\to y$ and hence $G((y,x))\geq\limsup_n
G((z_n,x))\geq\al$. On the other hand if $G((y,x))\geq\al$, then the sets
$\{v\in Y:G((v,x))<G((y,x))+\frac{1}{n}\}$ and
$\{z\in X:G((y,z))>\al-\frac{1}{n}\}$ are open for all $n$ and contain $y, x$
respectively. This completes the proof.
\end{proof}

Now we need the following technical Theorem which helps us to find a measurable selection of a closed set valued measurable function. We call a set valued map $F:X\to 2^Y$ is measurable if the graph of $F$, $Gr(F)=\{(x,F(x)):x\in X\}=\bigcup\{(x,y):x\in X, y\in F(x)\}\in \mc{B}_X\bigotimes\mc{B}_Y$. The last set represents the smallest $\si$-field containing the measurable rectangles $M\times N$, where $M\in\mc{B}_X, N\in\mc{B}_Y$, where $\mc{B}_X, \mc{B}_Y$ represent the Borel $\si$-fields over $X, Y$ respectively. 

\begin{thm}\cite[Corollary~5.5.8.]{SS}\label{T2}
Let $(\Omega,\mathfrak{M},\mu)$ be a complete probability space, $Y$ a polish space and $B\in \mathfrak{M}\bigotimes \mc{B}_Y$. Then $\pi_{\Omega}(B)\in \mathfrak{M}$ and $B$ admits a $\mathfrak{M}$ measurable section. 
\end{thm}

The above Theorem is a consequence of Von Naumann's selection Theorem (\cite[Theorem~5.5.2]{SS}); we may need to apply some other variant of this Theorem, but Theorem~\ref{T2} is crucially used in various places. 

\begin{lem}\label{L2}
Let $Y$ be a separable proximinal subspace of $X$ for which the map $P_Y:X\to 
2^Y$ is
uHsc. Let $f:I\to Y, g:I\to X$ are measurable, then for $\de>0$
consider the set valued function $\Phi_\de:I\to 2^Y$ defined by
$\Phi_\de(t)=P_{P_Y(g(t))}(f(t),\de)$. Then $\Phi_\de$ is measurable and it has
a measurable selection.
\end{lem}
\begin{proof}
Clearly we have $\Phi_\de(t)=P_Y(g(t))\cap B[f(t),\varphi(t)+\de]$, where
$\varphi$ is defined in Lemma~\ref{L1}. Since all functions in $\Phi_\de$ is
measurable, we have the graph $Gr(\Phi_\de)=\{(t,\Phi_\de(t)):t\in I\}$ is 
measurable. In fact we have the following representation for $\Phi_\de$.

Define $F_1, F_2:I\to 2^Y$ by $F_1(t)=B[f(t),\varphi(t)+\de]$ and 
$F_2(t)=P_Y(g(t))$. Since $f$ and $\varphi$ both the functions are measurable, 
$Gr(F_1)$ is measurable. Also $\{(t,y):t\in I, y\in 
F_2(t)\}=\{(t,y):\|y-f(t)\|=d(f(t),Y)\}=\bigcap_n\{(t,y):\|y-f(t)\|\leq 
\|y_n-f(t)\|\}$ where $(y_n)$ is a dense subset of $Y$. Hence the graph of 
$F_2$ 
is also measurable. Now $Gr(\Phi_\de)=Gr(F_1)\cap Gr(F_2)$. Hence 
$Gr(\Phi_\de)$ 
is again measurable. From Theorem~\ref{T2} it follows that the last 
set has a measurable selection.
\end{proof}

We now establish a distance formula between a given point in $L_p(I,Y)$ and 
the set of best approximation from a given point in $L_p(I,X)$ to $L_p(I,Y)$. 
Similar to Theorem~\ref{T1} the distance function is an integral of the point 
wise distance function.

\begin{thm}\label{L3}
Let $Y$ be a separable proximinal subspace of $X$ such that $P_Y$ is uHsc. Then
for $1\leq 
p<\iy$ and $f\in L_p(I,Y), g\in L_p(I,X)$,

$d(f, P_{L_p(I,Y)}(g))=\|d(f(.),P_Y(g(.)))\|_p$.
\end{thm}

\begin{proof}
From Lemma~\ref{L1} it follows that the map $t\mapsto 
d(f(t),P_Y(g(t)))$
is measurable and hence the above integral is justified. Now for the given range of $p$,
\beqa d(f,P_{L_p(I,Y)}(g)) 
&=& \inf_{h\in P_{L_p(I,Y)}(g)}\|f-h\|_p \\
                     &\geq & \|d(f(.),P_Y(g(.)))\|_p, \mbox{ from
Theorem~\ref{T1}(b)}.
\eeqa  

Now for each $n$ define $\Phi_n:I\to 2^Y$ by
$\Phi_n(t)=P_{P_Y(g(t))}(f(t),\frac{1}{n})$. From Lemma~\ref{L2} it follows that the
graph of $\Phi_n$ is measurable and hence by Theorem~\ref{T2} it has a 
measurable selection. Let
$h_n$ be such a selection. Clearly for all $t, h_n(t)\in P_Y(g(t))$ hence
$h_n\in P_{L_p(I,Y)}(g)$.

$d(f,P_{L_p(I,Y)}
(g))\leq\liminf_n\|f-h_n\|_p=\|d(f(.),
P_Y(g(.)))\|_p$. The last equality follows from the Dominated convergence
theorem for $p<\iy$ and this establishes the other inequality.
\end{proof}

\begin{rem}\label{R2}
For $p=\iy$, $P_{L_\iy(I,Y)}(g)\supseteq \{h\in L_\iy(I,Y):h(t)\in P_Y(g(t)) ~\mbox{a.e.}\}=Z$, say. Hence $d(f,P_{L_\iy(I,Y)}(g))\leq \|d(f(.),P_Y(g(.)))\|_\iy$~: In fact,
\beqa 
d(f,P_{L_\iy(I,Y)}(g))&\leq& d(f,Z) \\
                                           &=&\inf_{h\in Z}{ess \sup}_{t\in I}\|f(t)-h(t)\| \\
                                           &=& {\mbox{ess} \sup}_{t\in I} d(f(t),P_Y(g(t))) \\
                                           &=&\|d(f(.),P_Y(g(.)))\|_\iy.
\eeqa
\end{rem}

Our main results of this section are the following.

\begin{thm}\label{T6}
Let $Y$ be a separable proximinal subspace of $X$. Then $Y$ is strongly 
proximinal in $X$ if and only if $L_p(I,Y)$ is strongly proximinal in $L_p(I,X)$ for $1\leq p<\iy$.
\end{thm}
\begin{proof}
Let $Y$ be strongly proximinal in $X$ and let for some $p\in [1,\iy)$, $L_p(I,Y)$ be not strongly proximinal in 
$L_p(I,X)$. 
Hence there exists $f\in L_p(I,X), \e>0$ and $(g_n)\ci L_p(I,Y)$ such that 
$\|f-g_n\|_p\to d(f,L_p(I,Y))$ but $d(g_n,P_{L_p(I,Y)}(f))\geq\e$.

$ \mbox{Now } \|f-g_n\|_p\to d(f,L_p(Y)) $

$ \Longrightarrow \int_I\|f(t)-g_n(t)\|^pdm(t)\to \int_Id(f(t),Y)^pdm(t)$.

$ \Longrightarrow \int_I\left|\|f(t)-g_n(t)\|^p-d(f(t),Y)^p\right|dm(t)\to 0$.

A well known property of $L_p$ convergence ensures that there exists a subsequence $(g_{n_k})$ satisfying 
$\|f(t)-g_{n_k}(t)\|^p-d(f(t),Y)^p\to 0$ a.e.

Since $\|f(t)-g_{n_k}(t)\|\to d(f(t),Y)$ a.e. we have $d(g_{n_k}(t),P_Y(f(t)))\to 0$ a.e. Since $d(g_{n_k}(t),P_Y(f(t)))^p\leq 2\|f(t)\|^p$, a $L_1$ function. Hence by Dominated Converge Theorem, $\lim_{k\to 0}\int_I d(g_{n_k}(t),P_Y(f(t)))^pdm(t)= 0$, contradicting our assumption on $(g_n)$. Hence the result follows.
\end{proof}

Since all $g_n$'s in the above proof are separably valued the above proof can be fitted with all such strongly proximinal $Y$ of which all its separable subspaces are also strongly proximinal.

\begin{cor}\label{C2}
Let $Y$ be a stronly proximinal subspace of $X$. If every separable subspace of $Y$ is strongly proximinal in $X$ then $L_p(I,Y)$ is strongly proximinal in $L_p(I,X)$.
\end{cor}

\begin{proof}
 For such type of $(g_n)$ defined above get a separable subspace $Z\ci Y$ such that $d(f,L_p(I,Y))=d(f,L_p(I,Z))$, $1\leq p\leq\iy$. From our assumption and Theorem~\ref{T6} it follows $d(g_n,P_{L_p(I,Z)}(f))\ra 0$ and hence $d(g_n,P_{L_p(I,Y)}(f))\ra 0$.
\end{proof}

\begin{rem}
In general the conclusion of the Theorem~\ref{T6} is not true for $p=\iy$, Example~\ref{E2}. In next Section we show that a stronger version of strong proximinality of $L_p(I,Y)$ in $L_p(I,X)$ can be achieved from the similar assumption of $Y$ in $X$ and also vice versa.
\end{rem}

We now show that strong proximinality of $L_\iy(I,Y)$ in $L_\iy(I,X)$ demands a stronger assumption on $Y$ in $X$.

From Michael's selection theorem (see \cite[Theorem~3.1$^\prime$]{EM}) it is clear that if $Y$ is a finite dimensional subspace of a normed linear space $X$ and the metric projection $P_Y$ is lsc then it has a continuous selection. Now in \cite[Example~2.5]{AL} the author has shown that there exists a 1 dimensional subspace $Y$ in the 3 dimensional space $\mb{R}^3$ with a suitable norm where the metric projection $P_Y$ has no continuous selection. Hence it can not be lsc, and being a compact valued map $P_Y$ is not also lHsc. We now use these observations in the following example for the subspace $Y$ and the corresponding metric projection $P_Y$ to derive the non stability behavior of $L_\iy(I,Y)$ in $L_\iy(I,X)$ in the context of strong proximinality.

\begin{ex}\label{E2}
If $Y$ is strongly proximinal in $X$ then $L_\iy(I,Y)$ not necessarily strongly proximinal in $L_\iy(I,X)$~: Let $X$ and $Y$ be the spaces defined in \cite[Example~2.5]{AL}. Then there exists a sequence $(x_n)\ci X, x\in X$ such that $x_n\to x$ but $P_Y(x)\nsubseteq P_Y(x_n)+\e B_Y$ for some $\e>0$. Define $z_n=\frac{x_n}{d(x_n,Y)}, z_0=\frac{x}{d(x,Y)}$. Then $z_n\to z_0$ and $d(z_n,Y)=1=d(z_0,Y)$. Also we have,
\[
 d(x,Y)P_Y(z_0)\nsubseteq d(x_n,Y)P_Y(z_n)+\e B_Y, ~\mbox{for all~}n\in\mb{N}.
\]

That is there exists $y_n\in P_Y(z_0)$ such that $d\left((d(x,Y),d(x_n,Y)P_Y(z_n)\right)\geq\e$ and hence $d(y_n,\al_n P_Y(z_n))\geq\eta$ where $\al_n\to 1$ and some $\eta>0$. 

It is clear that $\left|\|y_n-z_n\|-d(z_n,Y)\right|\to 0$. Let $(I_n)$ be a sequence of pairwise disjoint intervals with $\cup_nI_n=I$. 

Define $f\in L_\iy(I,X), g_k\in L_\iy(I,Y)$ with $f|_{I_n}=z_n, g_k|_{I_n}=y_n$ if $k=n$ otherwise $g_k|_{I_n}\ci P_Y(z_k)$. Clearly we have $\|f-g_k\|_\iy\to d(f,L_\iy(I,Y))$ but $d(g_k,P_{L_\iy(I,Y)}(f))\geq\eta$, for all but finitely many $k$'s. The last inequality follows from the fact that, 
\[
P_{L_\iy(I,Y)}(f)=\{h\in L_\iy(I,Y):h|_{I_n}\ci P_Y(z_n),\mbox{ for all }n\}.
\]
\end{ex}

\begin{rem}\label{R3}
From above example it is clear if $L_\iy(I,Y)$ is strongly proximinal in $L_\iy(I,X)$ then $P_Y$ must be Hausdorff continuous.
\end{rem}

We conclude this Section by an application of Theorem~\ref{T1}. The scalar field for the Banach spaces considered in rest of this Section is $\mb{R}$.

The following result, Theorem~\ref{T4}, concludes about strong proximinality of $L_\iy(I,Y)$ in $L_\iy(I,X)$. It is also a strengthening of \cite[Theorem~15]{Rao}  which was proved for strong $1\frac{1}{2}$ ball property. Before we go for Theorem~\ref{T4} here is a 
useful characterization of $1\frac{1}{2}$ ball property.

\begin{thm}\label{T3}\cite{Go}
For a subspace $Y$ of $X$, \tFAE. \bla
\item $Y$ has $1\frac{1}{2}$ ball property.
\item $\|x-y\|=d(x,Y)+d(y,P_Y(x))$, for $x$ in $X$ and $y\in Y$.
\item $\|x\|=d(x,Y)+d(0,P_Y(x))$, for $x\in X$.
\el
\end{thm}

\begin{thm}\label{T4}
A separable subspace $Y$ of $X$ has $1\frac{1}{2}$ ball property if and only if $L_1(I,Y) (L_\iy(I,Y))$ has $1\frac{1}{2}$ ball property in $L_1(I,X) (L_\iy(I,X))$.
\end{thm}
\begin{proof} 
Suppose $Y$ has $1\frac{1}{2}$ ball property in $X$. We only show that the distance formula in Theorem~\ref{T3}$(c)$ holds for any $f\in L_1(I,X)$. Now $\|f(t)\|=d(f(t),Y)+d(0,P_Y(f(t)))$ a.e. For $p=1$, we get the result by integrating both sides and use the distance formulas discussed in Theorem~\ref{T1}, \ref{L3}. For $p=\iy$ we take the essential supremum in both sides and use the Remark~\ref{R2} and get $\|f\|_\iy\geq d(f,L_\iy(I,Y))+d(0,P_{L_\iy(I,Y)}(f))$. The other inequality is obvious.

Conversely, for any $x\in X$ consider the constant function $f(t)=x$ for all $t\in I$. The result now follows from Theorem~\ref{T3} and \ref{L3}.
\end{proof}

\section{Uniform proximinality of $L_p(I,Y)$ in $L_p(I,X)$}\label{S3}

In a recent paper (\cite{LZ}) the authors has introduced the notion {\it uniform proximinality} and it is claimed that closed unit ball of any uniformly convex space is uniformly proximinal. We first observe that the property does not holds even for the $2$ dimensional Euclidean space.

\begin{ex}\label{E5}
Let $C$ be the closed unit ball of $(\mb{R}^2,\|.\|_2)$, $x=(2,0)$. Then 
$P_C((2,0))=\{(1,0)\}$. Let $\al=2$ and $\e=1/2$. Then there does not exist 
$\de>0$ satisfying the condition in \cite{LZ}, pg 79, which makes $C$ 
{\it uniformly proximinal}. In fact, if such a $\de>0$ exists then 
$\|(0,0)-(2,0)\|<\al+\de$ but $\|(0,0)-(1,0)\|>\e$.
\end{ex}

We now define a stronger version of proximinality, viz. {\it uniform proximinality} which is in fact stated in \cite{PN} in the context of centres of closed bounded sets.

\bdfn\label{D2}
Let $C$ be a closed convex subset of $X$. We call $C$ is uniformly proximinal if given $\e>0$ and $R>0$ there exists $\de(\e,R)>0$ such that for any $x\in X, d(x,C)\leq R$ and $y\in C$ with $\|x-y\|<R+\de$, there exists $y^\prime\in C$ with $\|y-y^\prime\|<\e$ and $\|x-y^\prime\|\leq R$.
\edfn

Here are some examples of uniformly proximinal sets.
\begin{ex}\label{E1}
 \bla
\item It is clear that a Banach space $X$ having $3.2.I.P.$, 

$B_X (B_{L_\iy(I,X)})$ is uniformly proximinal in $X (L_\iy(I,X))$.
\item \cite[Proposition~3.5]{PN} Any $w^*$-closed convex subset of $\ell_1$ is uniformly proximinal.
\item \cite[Proposition~3.7]{PN} Any closed convex proximinal subset of a LUR space is uniformly proximinal.
\item Any subspace $Y$ of $X$ having $1\frac{1}{2}$ ball 
property is uniformly proximinal:~ Let $R, \e>0$ such that $d(x,Y)\leq R$ and $\|x-y\|<R+\e$ for some $y\in Y$, from the Definition~\ref{D4} we have $B[x,R]\cap B[y,\e]\cap Y\neq\es$. Any point from this intersection solve our purpose. 
\item \cite{Lau} Any subspace $Y$ of $X$ which is U-proximinal is also 
uniformly proximinal:~ Let $\eta, R>0$, suppose $\e:\mb{R}\to\mb{R}$ be the continuous function corresponding to the 
subspace $Y$ in \cite{Lau}. Get $\theta>0$ satisfying $\e(\theta)<\eta/R$, let 
$\de=R \theta$. Let $x\in X$ such that $d(x,Y)\leq R$ and $y\in Y$ be such that $\|x-y\|<R+\de$.

{\sc Claim:~}There exists $y^\prime\in Y$ such that $\|y-y^\prime\|<\eta$ and 
$\|x-y^\prime\|\leq R$.

Now $d(\frac{x}{R},Y)\leq 1$ and $\|\frac{x}{R}-\frac{y}{R}\|<1+\theta$, in 
other words $\frac{x}{R}\in Y+B_X$ and $\frac{x}{R}-\frac{y}{R}\in 
(1+\theta)B_X$ and hence $\frac{x}{R}-\frac{y}{R}\in Y+B_X$. And finally there 
exists $y_1\in \e(\theta)B_Y$ such that $\|\frac{x}{R}-\frac{y}{R}-y_1\|\leq 1$. 
Define $y^\prime=y+Ry_1$, this $y^\prime$ satisfies the desired requirements.
 \el
\end{ex}

We refer \cite{PN} to the reader for many other interesting uniformly proximinal 
subsets of Banach spaces. 

\begin{rem}\label{R9}
\bla
\item In the Definition~\ref{D2} if we demand to have $\de=\e$ for all $R>0$ we get back $1\frac{1}{2}$ ball property.
\item From the Definition~\ref{D2} it is clear that uniform proximinality  
of $C$ forces the set to be strongly proximinal. 
\item From the example by Godefroy in \cite{La} it is clear that the closed unit ball of a Banach space not necessarily have uniformly proximinal property.
\el
\end{rem}

We now claim that converse of Remark~\ref{R9}$(b)$ is not true. First observe the following.

\begin{prop}\label{P1}
If a closed convex set $C$ in $X$ is uniformly proximinal then the metric projection $P_C:X\to 2^C$ is continuous in the Hausdorff metric.
\end{prop}

\begin{proof}
Let $x_n\to x$ in $X$, without loss of generality we can assume $d(x,C)=1, d(x_n,C)=1$ for all $n$. Let $\de(1,\e)>0$ be the number corresponding to uniform proximinality of $C$. If possible let $P_C(x)\nsubseteq P_C(x_n)+\e B_Y$ for all but finitely many $n$'s, for some $\e>0$. Hence there exists $y_n\in P_C(x)$ such that $d(y_n,P_C(x_n))\geq \e$. Get a $N$ such that $|\|x_n-y_n\|-d(x_n,C)|<\de$ for all $n>N$. Now using the property of uniform proximinality of $C$ there exists $y_n^\prime\in P_C(x_n)$ such that $\|y_n-y_n^\prime\|<\e$, contradicting our hypothesis $d(y_n,P_C(x_n))\geq\e$. This proves $P_C$ is lHsc.

The uHsc of $P_C$ follows from strong proximinality of $C$.
\end{proof}

From Proposition~\ref{P1} and the arguments used before Example~\ref{E2}, it now follows that the subspace $Y$ in \cite[Example~2.5]{AL} can not be uniformly proximinal, while on the other hand being a finite dimensional subspace it is always strongly proximinal.

We now show that similar to proximinality and strong 
proximinality, the closed unit ball of a 
subspace by virtue of being uniformly proximinal forces the subspace to be uniformly proximinal.

\begin{prop}\label{P2}
For a subspace $Y$ of $X$, if $B_Y$ is uniformly proximinal then $Y$ is also uniformly proximinal.
\end{prop}

\begin{proof}
We use the technique used in \cite[Lemma~2.3]{BLR}. If possible let $B_Y$ is uniformly proximinal and $Y$ is not. From the definition there exist $R>0, \e>0, x\in X$ where $d(x,Y)\leq R$ and also there exists $(y_n)\ci Y$ such that $\|x-y_n\|<R+\frac{1}{n}$ but for all $y\in B(y_n,\e), \|x-y\|>R$.

Choose $\la>\|x\|+R+2\e$, then $d(x,\la B_Y)=d(x,Y)$. From our assumption on $y_n$ it follows that $\|y_n\|<\|x\|+R+\frac{1}{n}$ and hence $y_n\in \la B_Y$.

Uniform proximinality of $\la B_Y$ (and hence $B_Y$) would be contradicted if we can show that $B_Y(y_n,\e)\ci \la B_Y$, for all $n$. And It follows from the following observation.

$\|y_n\|+\e<\|x\|+R+\e+\frac{1}{n}\leq \|x\|+R+2\e<\la$, for large $n$.

This completes the proof.
\end{proof}

We now propose the following problem which is relevant to the subsequent matter.

\begin{problem}
Let $Y$ be a subspace of $X$ which is uniformly proximinal. Is it necessary that $B_Y$ is also uniformly proximinal in $X$ ?
\end{problem}

\begin{rem}\bla
%\item Converse of Proposition~\ref{P2} is not true. Recall the example by Saidi, 
\item It is clear from the Definition~\ref{D2} that uniform proximinality of 
$C$ is a uniform version of strong proximinality for the points which are of 
finite distance away from $C$. Hence due to the Example by Godefroy in \cite{La} 
it is clear that closed unit ball of a Banach space not necessarily uniformly 
proximinal.

\item We do not know whether the converse of Example~\ref{E1}(e) is true or not.

\item From Theorem~\ref{T4} we have if $Y$ is separable and also has $1\frac{1}{2}$ ball property in $X$ then $L_p(I,Y)$ has $1\frac{1}{2}$ ball property (hence uniformly proximinal) in $L_p(I,X)$ for $p=1, \iy$.
\el
\end{rem}

From the Definition~\ref{D2} we now have the following.

\begin{thm}\label{T14}
Let $Y$ be a separable proximinal subspace of $X$, Consider the following statements. \bla
\item $Y(B_Y)$ is uniformly proximinal in $X$.
\item $L_\iy(I,Y)(B_{L_\iy(I,Y)})$ is uniformly proximinal in $L_\iy(I,X)$.
\item $L_\iy(I,Y)(B_{L_\iy(I,Y)})$ is strongly proximinal in $L_\iy(I,X)$.
\el
Then $(a)\Longleftrightarrow (b)$ and $(b)\Longrightarrow (c)$.
\end{thm}

\begin{proof}
It is clear that $(b)\Longrightarrow (a)$ and $(b)\Longrightarrow (c)$. We only show that $(a)\Longrightarrow (b)$. We prove the result for the subspace $Y$, case for $B_Y$ follows from that with obvious modifications.

Let us choose $R>0$ and $\e>0$. Choose $\de(R,\e)>0$ for the subspace $Y$. We claim that this $\de$ will also work for $L_\iy(I,Y)$. Let $f\in L_\iy(I,X)$ with $d(f,L_\iy(I,Y))\leq R$. Let $g\in L_\iy(I,Y)$ be such that $\|f-g\|_\iy<R+\de$. Then from the property of uniform proximinality it follows that $B[f(t),R]\cap B[g(t),\e]\cap Y\neq\es$ a.e. Consider the set valued map $\varphi:t\mapsto B[f(t),R]\cap B[g(t),\e]\cap Y$ from $[0,1]$ to $2^Y$. It is clear that the graph of this map $\{(t,\phi(t):t\in I)\}$ is measurable and hence by Theorem~\ref{T2} it follows it has a measurable selection, let us call it $h$. We have $h\in L_\iy(I,Y)$ and satisfies the requirements.
\end{proof}

Theorem~\ref{T14} leads to the following problem.

\begin{problem}\label{P3}
Let $L_\iy(I,Y)$ is strongly proximinal in $L_\iy(I,X)$. Is it true that $Y$ is uniformly proximinal in $X$ ?
\end{problem}

\section{Ball Proximinality of $L_p(I,Y)$ in $L_p(I,X)$}\label{S4}

We first prove the distance formula analogous to Theorem~\ref{L3} for the closed 
unit ball of $L_p(I,Y)$, for $1\leq p\leq\iy$.

\begin{thm}\label{P10}
Let $f\in L_p(I,X)$ be a strongly measurable function then 

$d(f,B_{L_p(I,Y)})=\|d(f(.),{B_Y})\|_p$, for 
$1\leq p\leq\iy$.
\end{thm}
\begin{proof}
Case for $p=\iy$ is already observed in \cite{BLR}, it remains to prove when 
$p<\iy$.

{\sc Step~1:} Let $f(t)=x$ for all $t\in I$ and for some $x\in X$. Clearly 
$d(f,B_{L_p(I,Y)})\leq d(f,L_p(I,B_Y))=d(x,B_Y)$.

Let $g\in B_{L_p(I,Y)}$ and $\e>0$, then there is a sequence of simple functions 
$(s_n)\ci B_{L_p(I,Y)}$ such that $s_n\to g$ in $L_p(I,Y)$. Without loss of 
generality we may assume each $s_n$ has a following  representation. 
$s_n=\sum_{i=1}^{k_n}y_{i,n}\chi_{E_{i,n}}$. Where $y_{i,n}\in Y, 
\cup_iE_{i,n}=I$ and $E_{i,n}\cap E_{j,n}=\es$ for $i\neq j$.

Define $z_n=\sum_im(E_{i,n})y_{i,n}$, then $\|z_n\|^p\leq 
\sum_im(E_{i,n})\|y_{i,n}\|^p=\|s_n\|_p\leq 1$, first inequality follows from $x\mapsto \|x\|^p$ is a convex function. Hence $z_n\in B_Y$.

Now  
$d(x,B_Y)^p \leq \|x-z_n\|^p = \int_I\|f(t)-s_n(t)\|^pdm(t) =\|f-s_n\|_p^p\leq 
 \|f-g\|_p^p+\e$ for all but finitely many $n$'s. Taking infimum over $g\in 
B_{L_p(I,Y)}$ we get the result.

{\sc Step~2:} Let $f=\sum_{i=1}^nx_i\chi_{E_i}$, where $x_i\in X$, $\cup_iE_i=I$ 
and $E_i\cap E_j=\es$ for $i\neq j$.

 \beqa
\mbox{Now,~}&& d(f,B_{L_p(I,Y)})^p \\
&\leq & \int_Id(f(t),B_Y)^pdm(t) \\
                   &=& \sum_{1}^n d(x_i,B_Y)^pm(E_i) \\
									&=& 
\sum_{1}^n d(x_i,B_{L_p(I,Y)})^pm(E_i) \mbox{ follows from Step 1}\\
									&=& 
\inf_{g\in B_{L_p(I,Y)}} \sum_1^n \int_{E_i}\|x_i-g(t)\|^pdm(t) \\
									&=& 
\inf_{g\in B_{L_p(I,Y)}} \int_I \|f(t)-g(t)\|^pdm(t)=d(f,B_{L_p(I,Y)})^p
\eeqa

{\sc Step~3:} Let $f\in L_p(I,X)$ and $\e>0$. Get a sequence of simple 
functions $(s_n)\ci L_p(I,X)$ such that $s_n\to f$ in $L_p(I,X)$. Without loss 
of generality assume $s_n$ converges to $f$ pointwise and $\|s_n(t)\|\leq 
\|f(t)\|$ a.e.

\beqa
 \mbox{Now,~}&& d(f,B_{L_p(I,Y)}) \\
&=& \inf_{g\in B_{L_p(I,Y)}}\|f-g\|_p \\
&\geq & \inf_{g\in B_{L_p(I,Y)}}\|s_n-g\|_p-\|s_n-f\|_p \\
&=& d(s_n,B_{L_p(I,Y)})-\|s_n-f\|_p \\
&=& \left(\int_I d(s_n(t),B_Y)^pdm(t)\right)^{1/p}-\|s_n-f\|_p \mbox{; from {\sc Step~2}}\\
&\geq &\left(\int_I d(s_n(t),B_Y)^pdm(t)\right)^{1/p}-\e \mbox{; for large }n\\
&\geq &\left(\int_I d(f(t),B_Y)^pdm(t)\right)^{1/p}-2\e \mbox{; for large }n 
\eeqa
The last inequality follows from the following observation.
\beqa
\|d(f(.),B_Y)\|_p &\leq & \|d(f(.),B_Y)-d(s_n(.),B_Y)\|_p+\|d(s_n(.),B_Y)\|_p \\
                     &=& \left(\int_I \left|d(f(t),B_Y)-d(s_n(t),B_Y)\right|^p dm(t)\right)^{1/p}+ \\
										&& \hspace{4cm}  \|d(s_n(.),B_Y)\|_p \\
                     &\leq & \left(\int_I \|f(t)-s_n(t)\|^p dm(t)\right)^{1/p}+ \|d(s_n(.),B_Y)\|_p \\
                     \eeqa
Since $\e>0$ is arbitrary, the result follows.
\end{proof}

\begin{rem}\label{R1}
\bla
\item In \cite{BLR} it is observed that for $f\in L_p(I,X)$,  
$d(f,L_p(I,B_Y))=\|d(f(.),{B_Y})\|_p$, hence from Theorem~\ref{P10} it follows 
$P_{L_p(I,B_Y)}(f)\ci P_{B_{L_p(I,Y)}}(f)$ for $1\leq p\leq\iy$.
\item For a $g\in L_p(I,B_Y)$ we have, $g\in 
P_{B_{L_p(I,Y)}}(f)\Longleftrightarrow$ 

$g(t)\in P_{B_Y}(f(t)) \mbox{a.e.} 
\Longleftrightarrow g\in P_{L_p(I,B_Y)}(f)$ for $1\leq p<\iy$.
\el
\end{rem}

Remark~\ref{R1}$(a)$ leads to the following question.

\begin{problem}
For a subspace $Y$ of $X$ what are the functions $f\in L_p(I,X)$ for $1\leq p<\iy$ for which $P_{B_{L_p(I,Y)}}(f)=P_{L_p(I,B_Y)}(f)$ ?
\end{problem}

We now prove the main result of this Section.

\begin{thm}\label{T9}
Let $Y$ be a separable proximinal subspace of $X$. Then \tFAE.
\bla
\item $Y$ is ball proximinal in $X$.
\item $L_p(I,B_Y)$ is proximinal in $L_p(I,X)$, for $1\leq p\leq\iy$.
\item $L_p(I,Y)$ is ball proximinal in $L_p(I,X)$, for $1\leq p\leq\iy$.
\el
\end{thm}
\begin{proof}
From \cite{BLR} and Remark~\ref{R1} it is now clear that $(a)\Longrightarrow (b)$ and $(b)\Longrightarrow (c)$. We now show that $(c)\Longrightarrow (a)$. Now the 
Case for $p=\iy$ is already observed in \cite{BLR}, it remains to prove the result for $p<\iy$. Hence it is enough to prove that $Y$ is ball proximinal in $X$ if $L_p(I,Y)$ is same in $L_p(I,X)$ for some $p\in [1,\iy)$.

Let $x\in X$ and define $f(t)=x$ for all $t\in I$. Then $f\in L_p(I,X)$ and $d(f,B_{L_p(I,Y)})=d(x,B_Y)$. Choose $g\in B_{L_p(I,Y)}$ satisfying $\|f-g\|_p=d(x,B_Y)$. Now choose a sequence of simple functions $(s_n)$ such that $\|s_n-g\|_p\to 0$ where $\|s_n\|_p\leq \|g\|_p$. Let $s_n=\sum_{i=1}^{k_n}x_i^n\chi_{E_i^n}$ where $x_i^n\in Y$ and $\cup_iE_i^n=I$. Let $y_n=\sum_{i=1}^{k_n}x_i^nm(E_i^n)$. Since $\sum_{i=1}^{k_n}\|x_i^n\|^pm(E_i^n)\leq 1$ and $t\mapsto t^p$ is a convex function on $\mb{R}$ we have $y_n\in B_Y$. Now we have,
\beqa
d(x,B_Y)^p &\leq& \|x-y_n\|^p \\
            &=& \|x-\sum_{i=1}^{k_n}x_i^nm(E_i^n)\|^p \\
						&=& \|\sum_{i=1}^{k_n}(x-x_i^n)m(E_i^n)\|^p \\
						&\leq& \sum_{i=1}^{k_n} \|x-x_i^n\|^p m(E_i^n) \\
            &=& \|x-s_n\|_p^p \\
						&\ra& d(x,B_Y)^p
\eeqa
Which ensures that $(y_n)$ is a minimizing sequence in $B_Y$ for $x$. Clearly $(y_n)$ is cauchy; in fact $\lim_ny_n=\int_Ig(t)dm(t)$, and hence there exists $y_0\in B_Y$ such that $\|x-y_0\|=d(x,B_Y)$.
\end{proof}

The arguments involved in the proof of Corollary~\ref{C2} lead to the following conclusion.

\begin{cor}\label{C3}
\bla
\item Let $Y$ be a ball proximinal subspace of $X$, if every separable subspace of $Y$ is ball proximinal in $X$ then $L_p(I,Y)$ is ball proximinal in $L_p(I,X)$ for $1\leq p\leq\iy$.
\item Let $Y$ be a reflexive subspace of $X$ then $L_p(I,B_Y)$ (and hence $B_{L_p(I,Y)}$) is proximinal in $L_p(I,X)$ for $1\leq p\leq\iy$.
\el
\end{cor}
\begin{proof}
We only prove $(a)$, $(b)$ follows from $(a)$. It remains to prove for a given $f\in L_p(I,X)$, 
$P_{L_p(I,B_Y)}(f)\neq\es$. Choose $(g_n)\ci L_p(I,B_Y)$ such that $\|f-g_n\|_p\to d(f,L_p(I,B_Y))$. Get a separable subspace $Z\ci Y$ such that $g_n(I)\ci Z$ for all $n$. It is clear that $d(f,L_p(I,B_Y))=d(f,L_p(I,B_Z))$. Since $P_{L_p(I,B_Z)}(f)\neq\es$ the result follows.
\end{proof}

We now come to the strong proximinality of closed unit ball of $L_p(I,Y)$. A few 
routine modifications of Theorem~\ref{L3} lead to the following result.

\begin{thm}\label{T13}
Let $Y$ be a strongly ball proximinal subspace of $X$ and $f\in L_p(I,X),g\in
L_p(I,X)$ then, $d(f, P_{B_{L_p(I,Y)}}(g))=\|d(f(.),P_{B_Y}(g(.)))\|_p$, for $1\leq 
p<\iy$.
\end{thm}

Combining Theorem~\ref{T13} and the routine modifications in Theorem~\ref{T6}, 
one can have the following.
 
\begin{thm}\label{T12}
Let $Y$ be a separable proximinal subspace of $X$. Then \tFAE. 
\bla
\item $Y$ is strongly ball proximinal subspace of $X$.
\item $L_p(I,B_Y)$ is strongly proximinal in $L_p(I,X)$, for $1\leq p<\iy$.
\item $L_p(I,Y)$ is strongly ball proximinal in $L_p(I,X)$, for $1\leq p<\iy$.
\el
\end{thm}
\begin{proof}
It remains to prove $(c)\Longrightarrow (a)$. Choose $p\in [1,\iy)$ arbitrarily. Let $x\in X$ and $(y_n)\ci B_Y$ be such that $\|x-y_n\|\to d(x,B_Y)$. Define $f(t)=x$ and $g_n(t)=y_n$ for all $t\in I$ then $\|f-g_n\|_p\to d(f,B_{L_p(I,Y)})=d(x,B_Y)$ and hence $d(g_n,P_{B_{L_p(I,Y)}}(f))\to 0$. Choose $h_n\in P_{B_{L_p(I,Y)}}(f)$ such that $\|g_n-h_n\|_p\to 0$. Hence there exists $(z_n)\ci B_Y$ where $z_n=\int_Ih_n(t)dm(t)$.

{\sc Claim:~}  $z_n\in P_{B_Y}(x)$ and $\|y_n-z_n\|\ra 0$. 

\beqa
d(x,B_Y)^p\leq\|x-z_n\|^p &=& \|x-\int_Ih_n(t)dm(t)\|^p \\
          &=& \|\int_I(h_n(t)-x)dm(t)\|^p \\
					&\leq & \int_I \|h_n(t)-x\|^pdm(t) \\
					&=& \int_I d(x,B_Y)^p dm(t),  \mbox{~~follows from Theorem~\ref{T1}} \\
					&=& d(x,B_Y)
\eeqa

And finally,
\beqa
\|y_n-z_n\|^p &=& \|y_n-\int_I h_n(t)dm(t)\|^p \\
                                   &=& \|\int_I (y_n-h_n(t))dm(t)\|^p \\
																	 &\leq & \int_I \|y_n-h_n(t)\|^pdm(t) \\
                                   &\leq &\|g_n-h_n\|_p^p\to 0
\eeqa
 This completes the proof.
\end{proof}

For the case $p=\iy$ the result follows under an additional assumption on $B_Y$. The Banach spaces considered for rest of this Section are assumed to be Real.

Now it is clear from the above observations that,

\begin{cor}\label{C1}
Let $X$ be a separable Banach space.
\bla
\item For $1\leq p< \iy$, if $B_X$ is strongly proximinal in $X$ then $B_{L_p(I,X)}$ is stronly proximinal in $L_p(I,X)$.
\item If $X$ has $3.2.I.P.$ then $B_{L_p(I,X)}$ is stronly proximinal in $L_p(I,X)$ for $1\leq p<\iy$.
\el
\end{cor}

\begin{proof}
Since $X$ is separable, Theorem~\ref{T12} is true for $Y=X$ and hence $(a)$ follows. If $X$ has $3.2.I.P.$ then $B_X$ is strongly proximinal in $X$ (Example~\ref{E16}$(a)$). $(b)$ is now follows from $(a)$.
\end{proof}

\begin{rem}
\bla
\item Uniform convexity of $L_p(I,X)$ for $1< p< \iy$ follows from uniform convexity of $X$ and vice versa. Hence Corollary~\ref{C1} ensures the strong ball proximinality of $L_p(I,X)$ beyond the class of uniformly convex Banach space $X$.
\item It is not necessarily true that $B_{L_\iy(I,Y)}$ is strongly proximinal in $L_\iy(I,X)$ if $B_Y$ is same in $X$ (Example~\ref{E2}). 
\el
\end{rem}

\section{A new class of uniformly proximinal subsets}\label{S5}

Motivated from the property defined in Definition~\ref{D4} we define the following for a closed unit ball of a subspace but more generally it can be defined for a closed convex subset. 

\bdfn\label{D3}
We call the closed unit ball $B_Y$ of a subspace $Y$ in $X$ has $1\frac{1}{2}$ 
ball property if for $x\in X, y\in B_Y$ and $r_1, r_2>0$ $B[x,r_1]\cap 
B_Y\neq\es, \|x-y\|<r_1+r_2$ implies $B[x,r_1]\cap B[y,r_2]\cap B_Y\neq\es$.
\edfn

Similar to our earlier observation Remark~\ref{R9}$(a)$, the ball $B_Y$ having $1\frac{1}{2}$-ball property is uniformly proximinal for $\de=\e$.  Here are few immediate consequences of the above property.

\begin{thm}\label{T10}
 Let $Y$ be a subspace of $X$. Then, \bla
\item If $B_Y$ has $1\frac{1}{2}$ ball property then $Y$ has $1\frac{1}{2}$ ball 
property.
\item If $B_Y$ has $1\frac{1}{2}$ ball property in $X$ then $Y$ is ball proximinal in $X$.
 \el
\end{thm}

The proofs of the above Theorem follow from the similar arguments used to prove for a subspace for a similar claim. One can revisit the proofs in \cite[Proposition~2.4]{BLR} for $(a)$ and \cite[Lemma~1.1]{Y} for $(b)$.

\begin{rem}
 The converse of Theorem~\ref{T10}$(a)$ is not necessarily true. It is clear 
that a M-ideal has $1\frac{1}{2}$ ball property but not necessarily ball 
proximinal as is observed in \cite{JP}.
\end{rem}

We now derive a characterization, similar to Theorem~\ref{T3}, for $1\frac{1}{2}$ ball property of $B_Y$ in 
$X$. An almost similar arguments can be used to prove the following, for the sake of completeness we briefly outline it here.

\begin{notn}
For a subset $C$ of $X$, define $C_\e=\{x\in X:d(x,B)\leq\e\}$.
\end{notn}

\begin{thm}\label{T7}
 Let $Y$ be a subspace of $X$, then \tFAE. \bla
\item $B_Y$ has $1\frac{1}{2}$ ball property.
\item $P_{B_Y}(x,\de)={P_{B_Y}(x)}_\de\cap B_Y$. For all $x\in X$ and $\de>0$.
\item $d(y,P_{B_Y}(x))=\|y-x\|-d(x,B_Y)$. For all $x\in X, y\in B_Y$.
\el
\end{thm}
\begin{proof}
$(a)\Longrightarrow (b):$ Let $d=d(x,Y)$ and $\|x-y\|\leq d+\de$ for some $y\in B_Y$. By 
$(a)$, $B[x,d]\cap B[y,\de^\prime]\cap B_Y\neq\es$ for all $\de^\prime>\de$. 
That is $B[y,\de^\prime]\cap P_{B_Y}(x)\neq\es$ and hence 
$d(y,P_{B_Y}(x))\leq\de^\prime$, true for all $\de^\prime > \de$, thus 
$d(y,P_{B_Y}(x))\leq\de$. The other inclusion follows trivially from the 
definition of the sets involved in it.

$(b)\Longrightarrow (c):$ Let $\e=\|y-x\|-d(x,B_Y)$, for $y\in B_Y$. Then $y\in 
P_{B_Y}(x,\e)={P_{B_Y}(x)}_\e\cap B_Y$. Hence 
$d(y,P_{B_Y}(x))\leq\e=\|y-x\|-d(x,B_Y)$. The other inequality is obvious.

$(c)\Longrightarrow (a):$ Let $B[x,r_1]\cap B_Y\neq\es$ and $\|x-y\|<r_1+r_2$ 
for some $y\in B_Y$. Then $r_1=d+\de$ for some $\de\geq 0$, where $d=d(x,B_Y)$. 
If possible let $B[x,r_1]\cap B[y,r_2]\cap B_Y=\es$, that is 
$P_{B_Y}(x,\de)\cap B[y,r_2]=\es$. But then ${P_{B_Y}(x)}_\de\cap 
B[y,r_2]=\es$, that is $d(y,P_{B_Y}(x))>r_2+\de$. By $(c)$ $\|x-y\|-d>r_2+\de$ 
and finally $\|x-y\|>r_1+r_2$, a contradiction.
\end{proof}

We now show that the converse of Theorem~\ref{T10}$(a)$ is not true.

\begin{ex}
Consider the space $X=(\mb{R}^2,\|.\|_2)$ and let $Z=X\bigoplus_\iy \mb{R}$. Then $X$ is an M-ideal in $Z$ but for $x=((1,1),0)\in Z$, $\|x\|=\sqrt{2}$. Now for $y=((\frac{1}{\sqrt{2}},\frac{1}{\sqrt{2}}),1)\in B_Z$.
we have, $1=\|x-y\|<d(x,B_X)+d(y,P_{B_X}(x))=\sqrt{2}$ and hence from Theorem~\ref{T7} it follows that $B_X$ can not have $1\frac{1}{2}$ ball property in $Z$. 
\end{ex}

\begin{rem}\label{R15}
\bla
\item From the above characterizations it is clear that $1\frac{1}{2}$ ball property 
of $B_Y$ forces the subspace $Y$ to be strongly ball proximinal. 
\item From the example by Godefroy in \cite{La} it is clear that the closed unit ball of a Banach space not necessarily have $1\frac{1}{2}$ ball property.
\el
\end{rem}

Remark~\ref{R15}$(b)$ motivate us to investigate the class of Banach spaces and its subspaces whose closed unit balls are uniformly proximinal. The following examples are class of such spaces.

\begin{ex}\label{E16}
\bla
\item If $X$ has $3.2.I.P.$ then $B_X$ has $1\frac{1}{2}$ ball property in $X$, hence the closed unit ball of such a space is strongly proximinal. Hence for any real measure $\mu$, $L_1(\mu)$ or its isometric preduals have this property: Let $B[x,r]\cap B_X\neq\es$ and $\|x-z\|<r+s$ for some $z\in B_X$. The balls $B[x,r], B[z,s], B_X$ are pairwise intersecting and hence has non empty intersection.

\item Let $Y$ be a M-ideal in a $3.2.I.P$ space $X$ then $B_Y$ has 
$1\frac{1}{2}$-ball property in $X$:~
Let $B[x,r_1]\cap B_Y\neq\es$ and $\|x-y\|<r_1+r_2$ for some $y\in B_Y$. Hence 
we have $3$ balls $B[x,r_1], B[y,r_2], B_X$ in $X$ intersect pairwise. From the 
property of $3.2.I.P.$ we have $B[x,r_1]\cap B[y,r_2]\cap B_X\neq\es$. Now from 
\cite[Theorem~4.7]{JP} it follows $Y$ has strong $3$-ball property. Hence considering above 
$3$ balls once again one can have $B[x,r_1]\cap B[y,r_2]\cap B_X\cap Y\neq\es$ 
which in turn equivalent to $B[x,r_1]\cap B[y,r_2]\cap B_Y\neq\es$.
\el
\end{ex}

From the Definition~\ref{D3}, Theorem~\ref{T4} and the distance formulas proved in Theorem~\ref{P10}, \ref{T13}, we have,

\begin{cor}\label{C7}
Let $X$ be a separable Banach space. Then \tFAE.
\bla
\item $B_X$ has $1\frac{1}{2}$ ball property in $X$.
\item $B_{L_1(I,X)}$ has $1\frac{1}{2}$ ball property in $L_1(I,X)$.
\item $B_{L_\iy(I,X)}$ has $1\frac{1}{2}$ ball property in $L_\iy(I,X)$. 
\el
\end{cor}

\bibliographystyle{amsplain}

\begin{thebibliography}{99}

\bibitem{BLR} Pradipta Bandyopadhyay, Bor-Luh Lin, T. S. S. R. K. Rao, {\em 
Ball proximinality in Banach spaces}. Banach spaces and their applications in 
analysis, 251--264, Walter de Gruyter, Berlin, 2007.

\bibitem{AL} A. L. Brown, {\em Metric projections in spaces of integrable functions}, J. Approx. Theory {\bf 81}(1995){no. 1}, 78--103.

\bibitem{DN} S. Dutta and Darapaneni Narayana, \textit{Strongly proximinal 
subspaces in Banach spaces}, Contemp. Math. 435, Amer. Math. Soc., Providence, 
RI, 2007, 143--152.

\bibitem{DS} S. Dutta and P. Shunmugaraj, \textit{Strong proximinality of 
closed convex sets}, J. Approx. Theory 163 (2011), no. 4, 547-553.

\bibitem{GI} G. Godefroy and V. Indumathi, \textit{Strong proximinality and 
polyhedral spaces}, Rev. Mat. Complut. 14 (2001), no. 1, 105-125.

\bibitem{Go} G. Godini, {\em Best Approximation and Intersection of Balls}, 
Banach space theory and its applications, Lecture Notes in Mathematics, Vol 991, 1983, 44--54.
%\bibitem{IL} V. Indumathi, and S. Lalithambigai, {\em Ball proximinal spaces}, 
%J. Convex Anal. 18 (2011), no. 2, 353-366.

\bibitem{JP} C. R. Jayanarayanan and Tanmoy Paul, {\em Strong proximinality and intersection properties of balls in Banach spaces}, J. Math. Anal. Appl. {\bf 426}(2015), 2, 1217--1231.

\bibitem{Kh} R. Khalil, {\em Best Approximation in $L^p(I,X)$}, Math. Proc. 
Cambridge Phil. Soc. {\bf 94} 1983, 277--289.

\bibitem{DK} R. Khalil and W. Deeb, {\em Best Approximation in $L^p(I,X),II$} 
J. Approx. Theory (59) 1989, 296--299.

\bibitem{Lau} Ka-Sing Lau, {\em On a sufficient condition of proximity}, Trans. 
Amer. Math. Soc., {\bf 251} (1979) 343--356.

\bibitem{La} S. Lalithambigai, {\em Ball proximinality of equable space}, Collect Math. {\bf 60}, (2001), 79--88.

\bibitem{WL} W. A. Light, {\em Proximinality in $Lp(S,Y)$}, Rocky Mountain J. 
Math. {\bf 19}, (1989), 251--259.

\bibitem{LC} W. A. Light and E. W. Cheney, {\em Approximation Theory in Tensor Product Spaces} Lecture Notes in Mathematics {\bf 1169} Springer-Verlag

\bibitem{Lim} A. Lima, {\em Intersection properties of balls and subspaces in 
Banach spaces}, Trans. Amer. Math. Soc. 227 (1977), 1--62.

\bibitem{Lin} J. Lindenstrauss {\em Extension of Compact Operators}, Mem Amer. Math. Soc., vol 48, 1964.

\bibitem{LZ} Lin Pei-Kee, Zhang Wen, Zheng, Bentuo, {\em Stability of ball proximinality}, J. Approx. Theory, {\bf 183}, (2014), 72--81.

\bibitem{Men} J. Mendoza {\em Proximinality in $L_p(\mu,X)$}, J. Approx. 
Theory, {\bf 93}, (1998), 331--343.

\bibitem{EM} Ernest Michael, Continuous selection. I. Ann. of Math. (2) {\bf 63} (1956) 361--382.

\bibitem{PN} D. V. Pai and P. T. Nowroji, {\em On restricted centers of
sets}, J. Approx. Theory, {\bf 66}, 1991, 2, 170--189.

\bibitem{PY} R. Pay{\'a} and D. Yost, \textit{The two-ball property:
transitivity and examples}, Mathematika 35 (1988), no. 2, 190--197.

\bibitem{Rao} T. S. S. R. K. Rao, {\em The one and half ball property in spaces of vector-valued functions}, J. Convex Anal., {\bf 20}, (2013), 1, 13--23.

\bibitem{Sa} F. B. Saidi, {\em On the proximinality of the unit ball and of proximinal subspaces of Banach space: A counterexample}, Proc. Amer. Math. Soc., {\bf 133} (2005) no. 9, 2697--2703.

\bibitem{SS} S. M. Srivastava, A Course on Borel Sets, Grad. Texts in Math. 
180, Springer, New York, 1998.

\bibitem{Y} D. Yost \textit{Best approximation and intersection of balls in Banach spaces}, Bull. Austral. Math. Soc., {\bf 20}, (1979), 285--300 

\end{thebibliography}

\end{document}